\documentclass[12pt]{amsart}  

\usepackage{amsrefs}
\usepackage[all]{xy}
\usepackage{syntonly}
\usepackage{enumerate}
\usepackage{hyperref}
\usepackage{amsfonts}
\usepackage{amssymb}
\usepackage{amsmath}
\usepackage{amsthm}
\usepackage{enumitem}

\usepackage{tikz}

\usepackage{graphics}
\usepackage{graphicx}
\usepackage{color}
\usepackage{comment}

\newtheorem{theorem}{Theorem}
\newtheorem{corollary}[theorem]{Corollary}

\newtheorem{definition}[theorem]{Definition}
\newtheorem{lemma}[theorem]{Lemma}
\newtheorem{claim}[theorem]{Claim}

\newtheorem{question}[theorem]{Question}

\newtheorem{fact}[theorem]{Fact}

\newtheorem{NonGlobalClaim}{Claim}[theorem]  

\specialcomment{com}{ \color{red} }{\color{black}} 
\specialcomment{oldcom}{ \color{blue} }{\color{black}} 
\specialcomment{com2}{ \color{red} }{\color{black}} 

\excludecomment{com} 
\excludecomment{oldcom}
\excludecomment{com2}

\newcommand{\MyBoxWidth}{0.75}  

\begin{document}
\title[Martin's Maximum and tower forcing]{Martin's Maximum and tower forcing}

\author{Sean Cox and Matteo Viale}
\email{sean.cox@uni-muenster.de}
\address{
Institut f\"ur mathematische Logik und Grundlagenforschung \\
Universit\"at M\"unster \\
Einsteinstrasse 62 \\
48149 M\"unster \\
Tel.: +49-251-83-33 790 \\
Fax: +49-251-83-33 078
}



\begin{abstract}
There are several examples in the literature showing that compactness-like properties of a cardinal $\kappa$ cause poor behavior of some generic ultrapowers which have critical point $\kappa$ (Burke~\cite{MR1472122} when $\kappa$ is a supercompact cardinal;  Foreman-Magidor~\cite{MR1359154} when $\kappa = \omega_2$ in the presence of strong forcing axioms).  We prove more instances of this phenomenon.  First, the Reflection Principle (RP) implies that if $\vec{\mathcal{I}}$ is a tower of ideals which concentrates on the class $GIC_{\omega_1}$ of $\omega_1$-guessing, internally club sets, then $\vec{\mathcal{I}}$ is not presaturated (a set is $\omega_1$-guessing iff its transitive collapse has the $\omega_1$-approximation property as defined in Hamkins~\cite{MR2540935}).  This theorem, combined with work from \cite{VW_ISP}, shows that if $PFA^+$ or $MM$ holds and there is an inaccessible cardinal, then there is a tower with critical point $\omega_2$ which is not presaturated; moreover this tower is significantly different from the non-presaturated tower already known (by  Foreman-Magidor~\cite{MR1359154}) to exist in all models of Martin's Maximum.  The conjunction of the Strong Reflection Principle (SRP) and the Tree Property at $\omega_2$ has similar implications for towers of ideals which concentrate on the wider class $GIS_{\omega_1}$ of $\omega_1$-guessing, internally stationary sets.  

Finally, we show that the word ``presaturated'' cannot be replaced by ``precipitous'' in the theorems above: Martin's Maximum (which implies SRP and the Tree Property at $\omega_2$) is consistent with a precipitous tower on $GIC_{\omega_1}$.

\end{abstract}

\maketitle

\section{Introduction}

If the universe $V$ of sets satisfies ZFC, there is no elementary embedding $j: V \to N$ where $N$ is wellfounded and the least ordinal moved by $j$ is ``small'' (like $\omega_1$ or $\omega_2$).  Forcing with \emph{ideals} and \emph{towers of ideals} are two procedures that can potentially produce such an embedding $j: V \to N$ \emph{in some generic extension of $V$} where the least ordinal moved by $j$ is small.  A tower of ideals is a sequence of ideals $\vec{I} = \langle I_\lambda \ | \ \lambda < \delta \rangle$ with a certain coherence property (see Section \ref{sec_TowersMeasuresIdeals}).  The length of the sequence $\vec{I}$ is called the \emph{height of  $\vec{I}$} and, if each ideal in the sequence has the same completeness,\footnote{This will hold for all towers considered in this paper.} this completeness is called the \emph{critical point of $\vec{I}$}.  If $\vec{I}$ is a tower then there is a natural poset $\mathbb{P}_{\vec{I}}$ associated with $\vec{I}$, and in the generic extension $V^{\mathbb{P}_{\vec{I}}}$ there is an embedding $j_G: V \to ult(V,G)$ where the least ordinal moved by $j$ equals the critical point of the tower; this embedding is called a \emph{generic ultrapower of $V$ by $\vec{I}$} and $ult(V,G)$ is not necessarily wellfounded.    

Properties of the tower $\vec{I}$ and of its height affect the properties of the generic ultrapower.  Woodin proved that if $\delta$ is a Woodin cardinal, then many of the natural ``stationary towers'' of height $\delta$ satisfy a very strong property called \emph{presaturation} (see \cite{MR2069032} and \cite{MR1713438}).   Foreman and Magidor~\cite{MR1359154} proved that if $\delta$ is a supercompact cardinal then there are several natural stationary towers of height $\delta$ which are precipitous (a property weaker than presaturation).  For simplicity let us only consider towers with critical point $\omega_2$.  Then one key difference between the Foreman-Magidor stationary towers and the Woodin stationary towers are that with Woodin's examples, there are always some $V$-regular cardinals which become $\omega$-cofinal in the generic ultrapower; whereas they remain uncountably cofinal in the generic ultrapower in the Foreman-Magidor examples.  

On the other hand, compactness-like properties of the critical point of the tower can often \emph{prevent} nice behavior of the tower.  Burke~\cite{MR1472122} showed that, if $\kappa$ is supercompact and $\delta > \kappa$ is inaccessible, then there is a tower of height $\delta$ with critical point $\kappa$ which is not precipitous.  Strong forcing axioms like the Proper Forcing Axiom and Martin's Maximum are known to make $\omega_2$ behave much like a supercompact cardinal; so in light of Burke's theorem we should expect that strong forcing axioms prevent nice behavior of some towers with critical point $\omega_2$.  Foreman and Magidor~\cite{MR1359154} proved that Martin's Maximum\footnote{really, just the saturation of $NS_{\omega_1}$} implies that a certain natural tower\footnote{Namely, the stationary tower concentrating on the $\omega_1$-internally approachable structures.} with critical point $\omega_2$ is \emph{not} presaturated (see Example 9.16 of \cite{MattHandbook}).\footnote{However, their tower can be precipitous, and in fact always is precipitous if its height is a supercompact cardinal}  On a related note, they also showed that the Proper Forcing Axiom implies that there is \emph{no} presaturated ideal on $\omega_2$.     
  
This paper provides more results along the lines of the Burke and Foreman-Magidor theorems, that compactness properties of the critical point of certain towers prevents nice behavior of the tower.  We show that under strong forcing axioms, there are certain towers with critical point $\omega_2$ which are not presaturated; these towers are significantly different from the non-presaturated towers produced in Foreman-Magidor~\cite{MR1359154} in a very strong sense.\footnote{Namely, while the Foreman-Magidor tower concentrated on internally approachable structures, our ideals concentrate on structures which are definitely not internally approachable.}  Specifically:

\begin{theorem}\label{thm_MainGIC}
Assume $RP([\omega_2]^\omega)$ holds.  Whenever $\vec{I}$ is a tower which concentrates on the class $GIC_{\omega_1}$ of $\omega_1$-guessing, internally club sets, then $\vec{I}$ is not presaturated.
\end{theorem}
In fact we show more:  that $RP([\omega_2]^\omega)$ implies there is no \emph{single} ideal $J$ such that:
\begin{itemize}
 \item $J$ concentrates on $GIC_{\omega_1}$; and
 \item $J$ bounds its completeness.
\end{itemize} 
The latter is a property introduced by the first author in \cite{Cox_MALP} which is closely related to saturation and Chang's Conjecture.  See Definition \ref{def_IdealBoundsCompleteness} and Theorem \ref{thm_RP_Ideal}.  

Using Theorem \ref{thm_MainGIC} and results from \cite{VW_ISP}, we show:
\begin{theorem}\label{thm_PFAplus_impliesNonPresatGIC}
Assume either $PFA^+$ or $MM$, and that $\delta$ is inaccessible.  Then there is a tower of height $\delta$ with critical point $\omega_2$ which is not presaturated (in fact, this tower even fails to have the \emph{weak Chang Property}; see Definition \ref{def_weakChangProperty} and Corollary \ref{cor_PresatImpliesWeakChang}).
\end{theorem}

The hypotheses of Theorem \ref{thm_MainGIC} can be strengthened to obtain a stronger conclusion:

\begin{theorem}\label{thm_MainGIS}
Assume $SRP([\omega_2]^\omega)$ and  the Tree Property at $\omega_2$.  Whenever $\vec{I}$ is a tower which concentrates on the class $GIS_{\omega_1}$ of $\omega_1$-guessing, internally stationary sets, then $\vec{I}$ is not presaturated.
\end{theorem}
Again, similarly to Theorem \ref{thm_MainGIC} we actually show more:  that $SRP([\omega_2]^\omega)$ together with the Tree Property at $\omega_2$ implies there is no single ideal which concentrates on $GIS_{\omega_1}$ and bounds its completeness.

If we require the tower to be \emph{definable}, then a theorem of Burke~\cite{MR1472122} together with the Isomorphism Theorem from \cite{Viale_GuessingModel} yields:
\begin{theorem}\label{thm_NoDefinablePrecipGIC}
(ZFC):  If $2^\omega \le \omega_2$ then there is no precipitous tower of inaccessible height $\delta$ which concentrates on $GIC_{\omega_1}$ and is definable over $(V_\delta, \in)$
\end{theorem}

Finally, we prove that in Theorems \ref{thm_MainGIC} through \ref{thm_MainGIS}, the conclusion cannot be strengthened to say there is no \emph{precipitous} tower which concentrates on the relevant class of sets; even if the hypothesis is strengthened to ``plus'' versions of Martin's Maximum:
\begin{theorem}\label{thm_Consistent_GIC_Precip}
If $\kappa$ is supercompact and $\delta > \kappa$ is inaccessible, then if $\mathbb{P}$ is the standard iteration to produce a model of $MM^{+\omega_1}$, there is a precipitous tower in $V^{\mathbb{P}}$ of height $\delta$ concentrating on $GIC_{\omega_1}$.
\end{theorem}

The paper is organized as follows.  Section \ref{sec_Prelims} provides relevant background on guessing models (\ref{subsec_Prelim_GICetc}), forcing axioms and reflection principles (\ref{subsec_FA_ReflectionPrinciples}), the key Isomorphism Theorems for $\omega_1$-guessing structures from \cite{Viale_GuessingModel} (\ref{subsec_IsomThms}), towers of ultrafilters and ideals (\ref{sec_TowersMeasuresIdeals}), and induced towers (\ref{subsec_InducedTowers}).  Those last two subsections (\ref{sec_TowersMeasuresIdeals} and \ref{subsec_InducedTowers}) are used primarily for the consistency proof in Section \ref{sec_Consist_MM_GIC_Precip}, though Definition \ref{def_TowerIdeals} and Fact \ref{fact_TowerFacts} are used throughout the paper.  Section \ref{sec_TowersAndWeakChang} provides a brief review of the weak Chang property and relevant theorems from \cite{Cox_MALP} that will be used in the later proofs in Sections \ref{sec_RP_towers} and \ref{sec_SRP_towers} (but not in Section \ref{sec_Consist_MM_GIC_Precip}).  Section \ref{sec_RP_towers} proves Theorems \ref{thm_MainGIC}, \ref{thm_NoDefinablePrecipGIC}, and  \ref{thm_PFAplus_impliesNonPresatGIC}.  Section \ref{sec_SRP_towers} proves Theorem \ref{thm_MainGIS}.  The results in Sections \ref{sec_RP_towers} and \ref{sec_SRP_towers} are due to both authors.  Section \ref{sec_Consist_MM_GIC_Precip} proves Theorem \ref{thm_Consistent_GIC_Precip}, and is due to the first author.  Section \ref{sec_Questions} ends with some open problems.

\section{Preliminaries}\label{sec_Prelims}

\subsection{The classes \texorpdfstring{$GIC_{\omega_1}$}{GIComega1}, \texorpdfstring{$GIS_{\omega_1}$}{GISomega1}, and \texorpdfstring{$GIU_{\omega_1}$}{GIUomega1}}\label{subsec_Prelim_GICetc}
Weiss~\cite{Weiss_CombEssence} introduced the notion ISP, which is a significant strengthening of the Tree Property.  In this paper we use the alternative notion of a \emph{$\delta$-guessing model} from  \cite{Viale_GuessingModel}.  $ZF^-$ denotes $ZF$ without the Power Set Axiom.

\begin{definition}
Let $H$ be a transitive $ZF^-$ model and $\delta \in REG^H$.  We say that $H$ has the $\delta$-approximation property iff $(H,V)$ has the $\delta$-approximation property as in Hamkins~\cite{MR2540935}.  In other words, for every $\eta \in H$:  whenever $A \subset \eta$ is such that $z \cap A \in H$ for every $z \in H$ with $|z|^H < \delta$, then $A \in H$. 
\end{definition}
If $\omega_1 \subset M \prec H_\theta$ for some regular uncountable $\theta$, we say $M$ is \emph{$\omega_1$-guessing} iff its transitive collapse $H_M$ has the $\omega_1$-approximation property.\footnote{This definition is slightly different but equivalent to the definition in \cite{Viale_GuessingModel}.  Further, note that since we're assuming $M \prec H_\theta$ and $\omega_1 \subset M$, then $|z|^{H_M} = \omega$ iff $|z| = \omega$ (for any $z \in H_M$).}  We let $G_{\omega_1}$ denote the class of $M$ such that $|M| = \omega_1 \subset M$ and $M$ is  $\omega_1$-guessing.  $\sigma_M: H_M \to M$ will always denote the inverse of the Mostowski collapse of $M$.

We use several other common classes of structures (see Foreman and Todorcevic~\cite{MR2115072}).  A set $M$ is \emph{$\omega_1$-internally club} iff $M \cap [M]^\omega$ contains a club in $[M]^\omega$; \emph{$\omega_1$-internally stationary} iff $M \cap [M]^\omega$ is stationary in $[M]^\omega$; and \emph{$\omega_1$-internally unbounded} iff $M \cap [M]^\omega$ is $\subseteq$-cofinal in $[M]^\omega$.  Let $\Lambda := \{  M  \ | \ (M, \in) \text{ satisfies } ZF^- \text{, } \omega_1 \subseteq M \text{, and } |M| = \omega_1 \}$. We let $IC_{\omega_1}$, $IS_{\omega_1}$, and $IU_{\omega_1}$ refer respectively to the class of $M \in \Lambda$ which are $\omega_1$-internally club, $\omega_1$-internally stationary, and $\omega_1$-internally unbounded.  The classes $IC_{\omega_1}$, $IS_{\omega_1}$, and $IU_{\omega_1}$ can be equivalently characterized in ways analogous to internal approachability:
\begin{itemize}
 \item $M \in IA_{\omega_1}$ iff there is a $\subseteq$-continuous $\in$-chain $\langle N_\xi \ | \ \xi < \omega_1 \rangle$ such that $M = \bigcup_{\xi < \omega_1} N_\xi$ and $\langle N_\xi \ | \ \xi < \zeta \rangle \in M$ for every $\zeta < \omega_1$;
 \item $M \in IC_{\omega_1}$ iff there is a $\subseteq$-continuous $\in$-chain $\langle N_\xi \ | \ \xi < \omega_1 \rangle$ such that $M = \bigcup_{\xi < \omega_1} N_\xi$ and $N_\xi \in M$ for every $\xi < \omega_1$;
 \item $M \in IS_{\omega_1}$ iff there is a $\subseteq$-continuous $\in$-chain $\langle N_\xi \ | \ \xi < \omega_1 \rangle$ such that $M = \bigcup_{\xi < \omega_1} N_\xi$ and $N_\xi \in M$ for stationarily many $\xi < \omega_1$.  It is straightforward to see that $M \in IS_{\omega_1}$ iff there is some stationary $T_M \subset \omega_1$ such that for \emph{every} $\subseteq$-continuous $\in$-chain $\vec{N}$ with union $M$, $\{  \xi < \omega_1 \ | \ N_\xi \in M \} =_{NS} T_M$;\footnote{For $A,B \subset \omega_1$, $A =_{NS} B$ means that $\{  \xi < \omega_1 \ | \ \xi \in A \Delta B \}$ is nonstationary.}
 \item $M \in IU_{\omega_1}$ iff there is a $\in$-chain $\langle N_\xi \ | \ \xi < \omega_1 \rangle$ such that $M = \bigcup_{\xi < \omega_1} N_\xi$. 
\end{itemize}  
Set $GIC_{\omega_1}:=G_{\omega_1} \cap IC_{\omega_1}$, $GIS_{\omega_1}:=G_{\omega_1} \cap IS_{\omega_1}$, and $GIU_{\omega_1}:=G_{\omega_1} \cap IU_{\omega_1}$; note that $G_{\omega_1} \cap IA_{\omega_1}$ is always empty.\footnote{This is because if $M \in IA_{\omega_1}$ as witnessed by some sequence $\vec{N} = \langle N_\xi \ | \ \xi < \omega_1 \rangle$, then for every countable $z \in M$, $z \cap \vec{N} \in M$; if $M$ were in $G_{\omega_1}$ this would imply that $\vec{N} \in M$ and then that $M \in M$, which is of course impossible.}  Note also that all of the classes mentioned are invariant under isomorphism (i.e. $M$ is in the class iff its transitive collapse $H_M$ is in the class).  Viale and Weiss proved:
\begin{theorem}
(Viale-Weiss~\cite{VW_ISP})  $PFA$ implies that $GIC_{\omega_1} \cap \wp_{\omega_2}(H_\theta)$ is stationary for all regular $\theta \ge \omega_2$.  
\end{theorem} 
Their proof actually produced models which were not only $GIC_{\omega_1}$, but \emph{persistently so}; that is, these models remain in $GIC_{\omega_1}$ in any outer model which has the same $\omega_1$.  This used the following generalization of a theorem of Baumgartner.  For a poset $\mathbb{R}$ and a possibly non-transitive set $M$, let us say that a filter $g \subset \mathbb{R}$ is \emph{$(M, \mathbb{R})$-generic} iff $g \cap D \cap M \ne \emptyset$ for every $D \in M$ which is dense in $\mathbb{R}$.

\begin{theorem}\label{thm_GIC_persistent}
(Viale-Weiss~\cite{VW_ISP}) For each regular $\delta \ge \omega_2$ there is a proper poset $\mathbb{R}_\delta$ such that:
\begin{enumerate}
 \item $\mathbb{R}_\delta \in H_{\delta^+}$;
 \item $\Vdash_{\mathbb{R}} \check{H} \in GIC_{\omega_1}$ where $H:= H_\delta^V$;
 \item Whenever $M$ is a (possibly non-transitive) $ZF^-$ model such that $|M| = \omega_1 \subset M$ and there exists some $g$ which is $(M, \mathbb{R}_\delta)$-generic, then:
\begin{itemize}
 \item $M \cap H_{\delta} \in GIC_{\omega_1}$;
 \item\label{item_Persistence} If $W$ is any transitive $ZF$ model such that $(M,g, \mathbb{R}_\delta) \in W$ and $\omega_1^W = \omega_1^V$, then $W \models$ ``$M \in GIC_{\omega_1}$''.  (Here $W$ could, for example, be any outer model of $V$ which has the same $\omega_1$).
\end{itemize}
\end{enumerate}    
\end{theorem}

Viale proved:
\begin{lemma}\label{lem_MA_IU}
(Viale) $MA_{\omega_1}$ implies $G_{\omega_1} \subseteq IU_{\omega_1}$; so $G_{\omega_1} = GIU_{\omega_1}$.
\end{lemma}
\begin{proof}
Viale~\cite{Viale_GuessingModel} proved that if $M$ is $\omega_1$-guessing and $|M|$ is strictly less than the so-called \emph{pseudo-intersection number}, then $M \in IU_{\omega_1}$.  $MA_{\omega_1}$ implies that the pseudo-intersection number is $\ge \omega_2$. 
\end{proof}

Finally, we point out the standard fact that all of these classes \emph{project}:
\begin{lemma}\label{lem_ClassesProject}
Let $Z$ be any of the classes $G_{\omega_1}$, $IA_{\omega_1}$, $IC_{\omega_1}$, $IS_{\omega_1}$, or $IU_{\omega_1}$.  If $M \in Z$ and $\theta \ge \omega_2$ is a regular cardinal then $M \cap H_\theta \in Z$.
\end{lemma}
\begin{proof}
Here it will be more convenient to work with the following ``non-transitivised'' characterization of $G_{\omega_1}$:  $M \in G_{\omega_1}$ iff for every $\eta \in M$ and every $A \subset \eta \cap M$:  if $A \cap z \in M$ for every countable $z \in M$, then there is some $A' \in M$ such that $A' \cap M = A$.  

Now suppose $M \in G_{\omega_1}$ and $\theta \ge \omega_2$ is regular; we want to see that $M \cap H_\theta \in G_{\omega_1}$.  Let $\eta \in M \cap \theta$ and $A \subset \eta \cap M$, and suppose $z \cap A \in M \cap H_\theta$ for every countable $z \in M \cap H_\theta$.  Then clearly  $z \cap A \in M$ for every countable $z \in M$, since $A \subseteq \theta$.  Since $M \in G_{\omega_1}$, then there is an $A' \in M$ with $A' \cap M = A$.  Set $A'' := A' \cap \eta \in M \cap H_\theta$.  Then $A'' \cap (M \cap H_\theta) = A$.  
  
For the other classes, we present the argument for $IC_{\omega_1}$; the rest are similar.  Suppose $M \in IC_{\omega_1}$ as witnessed by a sequence $\langle N_\xi \ | \ \xi < \omega_1 \rangle$ where $N_\xi \in M$ for every $\xi < \omega_1$.  Let $\theta$ be a regular uncountable cardinal.  Clearly the sequence $\langle N_\xi \cap H_\theta \ | \ \xi < \omega_1 \rangle$ is $\subset$-increasing and $\subset$-continuous with union $M \cap H_\theta$; we just need to see that $N_\xi \cap H_\theta \in M \cap H_\theta$ for every $\xi < \omega_1$.  If $\theta \in M$ this is trivial.  If $\theta \notin M$ then since $M \in IC_{\omega_1}$, $M \cap ORD$ is an $\omega$-closed set of ordinals and so $sup(M \cap \theta)$ has uncountable cofinality.  Then for each $\xi < \omega_1$ there is some $\eta_\xi < \theta$, $\eta_\xi \in M$, such that $N_\xi \cap H_\theta = N_\xi \cap  H_{\eta_\xi}$; and the latter is in $M$ since both $\eta_\xi$ and $N_\xi$ are in $M$.    
\end{proof}

It is interesting to point out that by the argument of Proposition 2.4 of \cite{MR1359154}, if $Z$ is any of the classes $IA_{\omega_1}$, $IC_{\omega_1}$, $IS_{\omega_1}$, or $IU_{\omega_1}$, then $Z$ also \emph{lifts with respect to the nonstationary ideal}; that is, if $S$ is a stationary subset of $Z \cap \wp_{\omega_2}(H_\theta)$ and $\theta' >> \theta$, then $Z \cap \{  M \in \wp_{\omega_2}( H_{\theta'})  \ | \ M \cap H_\theta \in S \}$ is also stationary.  This implies that $\langle NS \upharpoonright (Z \cap \wp_{\omega_2}(H_\theta)) \ | \ \theta \in ORD \rangle$ forms a tower (see section \ref{sec_TowersAndWeakChang}).  On the other hand, this can trivially fail for the class $Z = G_{\omega_1}$ because $G_{\omega_1} \cap \wp_{\omega_2}(H_\theta)$ might be nonstationary for large $\theta$.  Even if $G_{\omega_1} \cap \wp_{\omega_2}(H_\theta)$ is stationary for every regular $\theta \ge \omega_2$ (as is the case under PFA), it is still not clear---and seems doubtful---that $\langle NS \upharpoonright G_{\omega_1} \cap \wp_{\omega_2}(H_\theta) \ | \ \theta \in ORD \rangle$ necessarily forms a tower.   

\subsection{Forcing Axioms, Projective Stationarity, and Reflection Principles}\label{subsec_FA_ReflectionPrinciples}

Let $\Gamma$ be a class of posets and $\beta$ an ordinal.  $FA^{+\beta}(\Gamma)$ means that for every $\mathbb{P} \in \Gamma$, for every $\omega_1$-sized collection $\mathcal{D}$ of dense subsets of $\mathbb{P}$, and for every sequence $\langle \dot{S}_\xi \ | \ \xi < \beta \rangle$ such that $\Vdash_{\mathbb{P}}$ ``$\dot{S}_\xi \subseteq \omega_1$ is stationary'' for every $\xi < \beta$, then there is a filter $F \subset \mathbb{P}$ meeting every $D \in \mathcal{D}$ and such that for every $\xi < \beta$:  $(\dot{S}_\xi)_F:= \{ \alpha < \omega_1 \ | \ (\exists q \in F)(q \Vdash \check{\alpha} \in \dot{S}_\xi)  \}$ is stationary.  $FA(\Gamma)$ means $FA^0(\Gamma)$ and $FA^+(\Gamma)$ means $FA^{+1}(\Gamma)$.  Martin's Axiom is $MA(\text{ccc posets})$, the Proper Forcing Axiom (PFA) is $MA(\text{proper posets})$, and Martin's Maximum (MM) is $MA(\text{posets preserving stationary subsets of } \omega_1)$.  We caution that elsewhere in the literature the notation $PFA^{++}$ and $MM^{++}$ are sometimes used for what we call $PFA^{+\omega_1}$ and $MM^{+\omega_1}$.  It is widely known that the standard iteration used to produce a model of $MM$ (resp. $PFA$) actually produces a model of $MM^{+\omega_1}$ (resp. $PFA^{+\omega_1}$).    

For a regular cardinal $\theta \ge \omega_2$, $RP([\theta]^\omega)$ means that whenever $S \subset [\theta]^\omega$ is stationary, then there is an $X$ such that $\omega_1 \subseteq X$, $|X| = \omega_1$, and $S \cap [X]^\omega$ is stationary in $[X]^\omega$.  It is well-known that $FA^+(\sigma\text{-closed})$ implies $RP([\theta]^\omega)$ for all regular $\theta \ge \omega_2$ (so in particular RP follows from $PFA^+$); and by \cite{MR924672} this is also implied by MM.  A set $P \subset [X]^\omega$ is \emph{projective stationary} iff for every stationary $T \subset \omega_1$, $\{ Y \in P \ | \ Y \cap \omega_1 \in T   \}$ is stationary in $[X]^\omega$; equivalently, the projection of $P$ to $\omega_1$ contains a club subset of $\omega_1$.  For $\theta \ge \omega_2$, the \emph{Strong Reflection Principle at $\theta$ ($SRP(\theta)$)} is the statement:  for every projective stationary $P \subset [H_\theta]^\omega$, there is a continuous elementary chain $\langle N_\xi \ | \ \xi < \omega_1 \rangle$ of countable models such that every $N_\xi$ is an element of $P$ (i.e. there is some $\omega_1$-sized subset $X$ of $\theta$ such that $P \cap [X]^\omega$ contains a club in $[X]^\omega$).  It was shown in \cite{MR1668171} that Martin's Maximum implies $SRP(\theta)$ for all regular $\theta \ge \omega_2$.  Extending a result of Gitik~\cite{MR820120}, Velickovic proved the following theorem (see Section 3 of \cite{MR1174395}):
\begin{theorem}\label{thm_Velick}
Whenever $C \subset [\omega_2]^\omega$ is club, $x \in \mathbb{R}$, and $T \subset \omega_1$ is stationary, then there are $a,b,c \in C$ such that $a \cap \omega_1 = b \cap \omega_1 = c \cap \omega_1 \in T$ and $x  \in L_{\omega_2}[a,b,c]$.
\end{theorem}
\begin{corollary}\label{cor_ToVelick}
If $W$ is a transitive $ZF^-$ model with $\omega_2 \subseteq W$ and $\mathbb{R} - W \ne \emptyset$, then $[\omega_2]^\omega - W$ is projective stationary.
\end{corollary}
\begin{proof}
Let $T \subset \omega_1$ be stationary; we need to show that $\{ d \in [\omega_2]^\omega \ | \ d \notin W  \text{ and } d \cap \omega_1 \in T \}$ is stationary in $[\omega_2]^\omega$.  Suppose not; then there is a club $C \subset [\omega_2]^\omega$ such that $C \searrow T := \{ d \in C \ | \ d \cap \omega_1 \in T  \} \subset W$.  Let $x \in \mathbb{R}$ be arbitrary and let $a,b,c \in C$ be as in Theorem \ref{thm_Velick}, so that $x \in L_{\omega_2}[a,b,c]$ and $a \cap \omega_1 = b \cap \omega_1 = c \cap \omega_1 \in T$; so $a,b,c \in C \searrow T \subset W$.  Since $\omega_2 \subseteq W \models ZF^{-}$ and $a,b,c \in W$ then $L_{\omega_2}[a,b,c] \subseteq W$.  So $x \in W$; since $x$ was arbitrary we've shown $\mathbb{R} \subset W$, contrary to the assumptions.
\end{proof}

\subsection{\texorpdfstring{Isomorphism Theorems for $GIC_{\omega_1}$ and $GIS_{\omega_1}$}{Isomorphism Theorems for GIComega1 and GISomega1}}\label{subsec_IsomThms}

We use the Isomorphism Theorems from Viale~\cite{Viale_GuessingModel}.  For transitive $ZF^-$ models $M$ and $M'$, we say that \emph{$M$ is a hereditary initial segment of $M'$} iff $M= M'$ or there is some $\lambda \in Card^{M'}$ such that $M = (H_{\lambda})^{M'}$.  For possibly non-transitive $M$ and $M'$, we say $M$ is a hereditary initial segment of $M'$ iff this holds for their transitive collapses.    
\begin{theorem} \label{thm_IsomIS}
(Viale) Assume $2^\omega = \omega_2$.  Let $\theta$ be a regular uncountable cardinal $\ge \omega_2$ and $\Delta$ a wellorder on $H_\theta$.  Suppose $M$ and $M'$ are submodels of $(H_\theta, \in, \Delta)$ such that $M \cap \omega_2 = M' \cap \omega_2$.

\begin{itemize}
 \item If $M$ and $M'$ are $GIC_{\omega_1}$, then one is a hereditary initial segment of the other.
 \item If $M$ and $M'$ are $GIS_{\omega_1}$, $T_M \subset \omega_1$ and $T_{M'} \subset \omega_1$ witness that $M, M' \in IS_{\omega_1}$ (respectively), and $T_M \cap T_{M'}$ is stationary, then one of $M, M'$ is a hereditary initial segment of the other.
\end{itemize}
\end{theorem}

\subsection{Towers of measures and towers of ideals}\label{sec_TowersMeasuresIdeals}

Suppose $Z$ is a set and $F \subset \wp(Z)$ is a filter.  The \emph{support of $F$ ($supp(F)$)} is the set $\bigcup Z$.  For all instances in this paper, the support of a filter will always be a transitive set (typically some $H_\theta$) and $Z$ will always be of the form $\wp_\kappa(H_\theta)$ for some regular $\kappa \le \theta$.  \emph{(Ultra) filter} will always mean a normal,\footnote{$F$ is normal iff for every regressive $g: Z \to V$ there is an $S \in F^+$ such that $g \upharpoonright S$ is constant.} countably complete, fine\footnote{i.e. for every $b \in supp(F)$ there is an $A \in F$ such that $b \in M$ for all $M \in A$.  Note if $F$ is fine then its support is equal to $\bigcup \bigcup F$.} (ultra) filter.  If $F$ is a filter then $\breve{F}$ denotes its dual ideal; similarly if $I$ is an ideal then $\breve{I}$ denotes its dual filter.  If $\Gamma$ is a class, we say that a filter $F$ \emph{concentrates on $\Gamma$} iff there is an $A \in F$ such that $A \subseteq \Gamma$; if $I$ is an ideal we say that $I$ \emph{concentrates on $\Gamma$} iff its dual filter concentrates on $\Gamma$.  A set $S \subset  \bigcup I$ is \emph{$I$-positive} (written $S \in I^+$) iff $S \notin I$.  
\begin{definition}\label{def_CanonicalProjection}
If $Z \subseteq Z'$, $I \subset \wp(Z)$ and $I \subseteq \wp(Z')$ are ideals, we say that \emph{$I$ is the canonical projection of $I'$ to $Z$} iff $I = \{ \{ M' \cap Z \ | \ M' \in A'   \}  \ | \ A' \in I'    \}$.
\end{definition}
Suppose $W$ is a transitive model of set theory and $U$ is a (possibly external) $W$-normal,\footnote{i.e. normal with respect to sequences from $W$} fine ultrafilter; say $U \subset \wp(Z)$ where $Z \in W$ (for example, $Z$ might be $\wp^W_\kappa(H^W_\lambda)$).  It is a standard fact that if $j_U: W \to_U ult(W,U) = ({}^{Z} W \cap W) /U$ is the ultrapower and $\bigcup Z$ is transitive, then:
\begin{equation}\label{eq_ExternalUltrapowerFacts}
\parbox{\MyBoxWidth \linewidth}{
\begin{itemize}
 \item $j_U " (\bigcup Z)$ is an element of $ult(W,U)$ and is represented by $[id_{\bigcup Z}]_U$;
 \item If the wellfounded part of $ult(W,U)$ has been transitivized, then $j \upharpoonright (\bigcup Z)$ is an element of $ult(W,U)$ and is represented by $[f]_U$ where $f(M):\simeq $ the inverse of the Mostowski collapse map of $M$. 
\end{itemize}  
}
\end{equation}
Of course if $U \in W$ then $ult(W,U)$ is wellfounded, but the comments above show that $\bigcup Z$ is always an element of the (transitivised) wellfounded part of $ult(W,U)$, even when $U$ is external to $W$.  One common example of this ``external'' case is generic ultrapowers.  Suppose $I \subset \wp(Z)$ is an ideal \footnote{recall we are assuming all ideals are normal, fine, and countably complete.}, and let $\mathbb{P}_I := (I^+, \subset)$.  If $G$ is $(V, \mathbb{P}_I)$-generic then $G$ is an ultrafilter on $\wp^V(Z)$ which is normal\footnote{By a density argument and the fact that $I$ was a normal ideal.} with respect to sequences from $V$.  In particular, (\ref{eq_ExternalUltrapowerFacts}) holds and $j_G \upharpoonright (\bigcup Z) \in ult(V,G)$.

Now consider generalizations of these notions to sequences of filters which cohere via the ``canonical projection'' relation in Definition \ref{def_CanonicalProjection}.
\begin{definition}\label{def_TowerOfExternalMeasures}
Let $W$ be a transitive model of set theory, and $\delta$ a regular cardinal in $W$.  Let $\langle Z_\lambda \ | \ \lambda < \delta \rangle \in W$ and for simplicity, assume each $\bigcup Z_\lambda \in V^W_\delta$ and each $\bigcup Z_\lambda$ is transitive, and $\bigcup_{\lambda < \delta} \bigcup Z_\lambda = V^W_\delta$.  Suppose $\langle U_\lambda \ | \ \lambda < \delta \rangle$ is a (possibly external to $W$) sequence of $W$-normal ultrafilters, where $U_\lambda \subset \wp(Z_\lambda)$ for each $\lambda < \delta$.  Also assume there is a fixed $\kappa < \delta$ such that each $U_\lambda$ has completeness $\kappa$.  We will call $\vec{U}$ a \emph{tower of $W$-normal measures} iff for every $\lambda \le \lambda' < \delta$:  $U_\lambda$ is the canonical projection of $U_{\lambda'}$ to $Z_\lambda$ (as in Definition \ref{def_CanonicalProjection}).
\end{definition}

If $\vec{U}$ is a tower of $W$-normal measures, then there is a commutative system of maps obtained by the various ultrapower maps $j_{U_\lambda}: W \to_{U_\lambda} ult(W,U_\lambda)$ and for $\lambda \le \lambda'$, maps $k_{\lambda, \lambda'}: ult(W,U_\lambda) \to ult(W,U_{\lambda'})$ given by $[f]_{U_\lambda} \mapsto [M' \mapsto f(M' \cap \bigcup Z_\lambda)]_{U_{\lambda'}}$.  The direct limit map of the system is denoted $j_{\vec{U}}: W \to_{\vec{U}} N_{\vec{U}}$.  If $\vec{U} \in W$ then this direct limit will always be wellfounded and closed under $< \delta$ sequences from $W$; so if in addition $j_{\vec{U}}(\kappa) = \delta$ then $j_{\vec{U}}$ can witness the almost-hugeness of $cr(j_{\vec{U}})$ in $W$.\footnote{See Theorem 24.11 of Kanamori~\cite{MR1994835} for technical criteria on $\vec{U}$ which will guarantee that $j_{\vec{U}}$ is an almost huge embedding.}  

A (possibly external) direct limit embedding $j_{\vec{U}}: W \to_{\vec{U}} N_{\vec{U}}$ can also be viewed as an ultrapower embedding as follows.  Given a (partial) function $f: V^W_\delta \to W$ with $f \in W$, let $supp(f)$ denote the least cardinal $\lambda \le \delta$ such that $f(x)$ only depends on $x \cap H_\lambda$.  Let $B^W_{< \delta} := \{  f \in W \ | \ f: V_\delta^W \to W \text{ and } \ supp(f) < \delta \}$.  Define an equivalence relation $\simeq_{\vec{U}}$ on $B^W_{<\delta}$ by:  $f \simeq_{\vec{U}} g$ iff $\{  M \in Z_\lambda \ | \ f(M \cap H_\lambda) = g(M \cap H_\lambda) \} \in U_\lambda$ for all sufficiently large $\lambda < \delta$.  Define a relation $\in_{\vec{U}}$ on $B^W_{<\delta}/\simeq_{\vec{U}}$ in the obvious way (this will be well-defined).  Then the direct limit $(N_{\vec{U}}, E_{\vec{U}})$ will be isomorphic to $(B^W_{<\delta}/\simeq_{\vec{U}}, \in_{\vec{U}})$; for this reason we will write $ult(W, \vec{U})$ for this direct limit.  the Los Theorem will hold in the following form:  for each $f_0, ..., f_n$ in $B^W_{<\delta}$ and each formula $\phi$:  $N_{\vec{U}} \models \phi([f_0]_{\vec{U}}, ..., [f_n]_{\vec{U}})$ iff $ \{ M \in Z_\lambda \ | \ W \models \phi(f_0(M), ..., f_n(M))  \} \in U_\lambda$ for every sufficiently large $\lambda < \delta$.  The following analogues of (\ref{eq_ExternalUltrapowerFacts}) always hold when taking (possibly external) ultrapowers by a tower of $W$-normal measures:
\begin{equation}\label{eq_AlwaysHoldExternalTowerUltrapower}
\parbox{\MyBoxWidth \linewidth}{
For every $X \in V^W_\delta$:
\begin{itemize}
 \item $j_{\vec{U}} " X$ is an element of $ult(W,\vec{U})$ and is represented by the function $[M \mapsto M \cap X]_{\vec{U}}$
 \item $j_{\vec{U}} \upharpoonright X$ is an element of $ult(W, \vec{U})$ and is represented by $M \mapsto$ the inverse of the Mostowski collapse map of $M \cap X$.
\end{itemize}
}
\end{equation}  

Just as forcing with the positive sets of an ideal gives rise to external ultrapowers of $V$ by a single $V$-normal measure, forcing with a \emph{tower of ideals} (defined below) gives rise to an external ultrapower of $V$ by a tower of $V$-normal measures. 

\begin{definition}\label{def_TowerIdeals}
A sequence $\langle I_\lambda \ | \ \lambda < \delta \rangle$ is called a \emph{tower of ideals of height $\delta$} iff for every $\lambda \le \lambda' < \delta$, $I_\lambda$ is the canonical projection (in the sense of Definition \ref{def_CanonicalProjection}) of $I_{\lambda'}$ to $Z_\lambda$.

We will also require for simplicity that for each $\lambda$, if $Z_\lambda$ is such that $I_\lambda \subset \wp(Z_\lambda)$, then $\bigcup Z_\lambda = H_\lambda$.  In this paper $Z_\lambda$ will always be of the form $\wp_\kappa(H_\lambda)$.   

For a class $\Gamma$, we say that $\vec{I}$ \emph{concentrates on $\Gamma$} iff every ideal in the sequence concentrates on $\Gamma$.
\end{definition}

If $\vec{I}$ is a tower, there is a natural poset $\mathbb{P}_{\vec{I}}$ associated with $\vec{I}$.  Conditions are pairs $(\lambda, S)$ where $\alpha < \delta$ and $S \in I_\lambda^+$.  A condition $(\lambda, S)$ is strengthened by increasing $\lambda$ to some $\lambda'$ and refining the lifting of $S$ to $H_{\lambda'}$.  More precisely:  $(\lambda', S') \le (\lambda, S)$ iff $\lambda' \ge \lambda$ and $S' \subseteq S^{Z_{\lambda'}}:= \{ M' \in Z_{\lambda'} \ | \ M' \cap H_\lambda \in S  \} $.  If $G$ is generic for $\mathbb{P}_{\vec{I}}$, then let $proj(G,\lambda):=\{ S \in \wp^V(Z_\lambda) \ | \ (\lambda,S) \in G   \}$; this is an ultrafilter on $\wp^V(Z_\lambda)$ which is normal with respect to sequences from $V$ (though $proj(G,\lambda)$ need \emph{not} be $(V, \mathbb{P}_{I_\lambda})$-generic!) and $\langle proj(G,\lambda)\ | \ \lambda < \delta \rangle$ is a tower of $V$-normal measures as in Definition \ref{def_TowerOfExternalMeasures}; in particular (\ref{eq_AlwaysHoldExternalTowerUltrapower}) holds and one can prove the following general facts (in the case of towers, we use the notation $j_G: V \to_G ult(V,G)$ to denote the ultrapower embedding $j_{\vec{U}}: V \to_{\vec{U}} ult(V, \vec{U})$ where $\vec{U} = \langle proj(G,\lambda) \ | \ \lambda < \delta \rangle$):

\begin{fact}\label{fact_TowerFacts}
If $\vec{I}$ is a tower of height $\delta$ where $\delta$ is inaccessible and $G$ is generic for $\vec{I}$, then:
\begin{enumerate}
 \item\label{item_MapRestrictInUlt} For every $D \in V_\delta$, $j_G \upharpoonright D \in ult(V,G)$
 \item For every $\theta \in U$:  $proj(G, \theta) = \{ S \in \wp^V(Z_\theta) \ | \ j_G `` H_\theta \in j_G (S)  \}$.  This fact, combined with item \ref{item_MapRestrictInUlt} and the assumption that $\delta$ is (strongly) inaccessible, implies that $proj(G, \theta) \in ult(V,G)$ for every $\theta < \delta$. 
 \item For every $\theta < \delta$ and every $Y \in V_\delta$:  the relations $=_{proj(G,\theta)} \upharpoonright ({}^{H_\theta} Y )^V$ and $\in_{proj(G,\theta)} \upharpoonright ({}^{H_\theta} Y )^V$ are elements of $ult(V,G)$. (this follows from the previous bullets:  both $({}^{H_\theta} Y)^V$ and $proj(G,\theta)$ are elements of $ult(V,G)$).
 \item\label{item_k_maps} If $\vec{I}$ concentrates on $\{ M \ | \ \lambda \subset M  \text{ and } M \cap \lambda^+ \in \lambda^+ \}$ and $j_{proj(G,\theta)}: V \to ult(V,proj(G,\theta))$ is the ultrapower map by $proj(G,\theta)$, then $k_{proj(G,\theta),proj(G,\theta')} \upharpoonright j_{proj(G,\theta)}(\lambda^+) = id$.
\end{enumerate}
\end{fact}

We refer the reader to Foreman~\cite{MattHandbook} for the general theory of towers, and to Larson~\cite{MR2069032} and Woodin~\cite{MR1713438} for the specific cases where all the ideals $I_\theta$ in the tower are of the form $NS \upharpoonright Z_\theta$ (towers of this form are called \emph{stationary towers}).

\subsection{Induced towers of ideals}\label{subsec_InducedTowers}

We adjust Example 3.30 from \cite{MattHandbook} to towers:
\begin{definition}\label{def_TowerDerivedFromName}
Suppose $\mathbb{Q}$ is a poset, $\delta$ is inaccessible, and $\langle \dot{U}_\lambda \ | \ \lambda < \delta \rangle$ is a sequence of $\mathbb{Q}$-names such that $\mathbb{Q} \Vdash$ ``$\vec{\dot{U}}$ is a tower of $V$-normal ultrafilters''.  For each $\lambda < \delta$, let $I_\lambda$ be the collection of $A$ such that for every $(V, \mathbb{Q})$-generic object $H$, $A \notin \dot{U}_H$.  The sequence $\langle I_\lambda \ | \ \lambda < \delta \rangle$ will be called the \emph{tower of ideals derived from the name $\vec{\dot{U}}$}.  
\end{definition}
It is straightforward to check that this indeed forms a tower of ideals.

Recall that if $j: V \to N$ is an embedding with critical point $\kappa$ and $\mathbb{P} \in V$ is a poset such that $j \upharpoonright \mathbb{P}: \mathbb{P} \to j(\mathbb{P})$ is a \emph{regular embedding},\footnote{i.e. whenever $A \subset \mathbb{P}$ is a maximal antichain then $j[A]$ is a maximal antichain in $j(\mathbb{P})$.} then $j(\mathbb{P})$ is forcing equivalent to $\mathbb{P} * j(\mathbb{P})/j[\dot{G}]$ where $\dot{G}$ is the canonical $\mathbb{P}$-name for the $\mathbb{P}$-generic.  Further, whenever $G*H$ is generic for $\mathbb{P} * j(\mathbb{P})/j[\dot{G}]$ then $j$ can be lifted (in $V[G][H]$) to an elementary $j^{G*H}: V[G] \to N[G][H]$.  Suppose $\delta$ is a $V$-cardinal such that for every $\lambda < \delta$, $j " H_\lambda \in N$.  Then for every $\lambda < \delta$:
\begin{equation}\label{eq_U_GstarH}
\parbox{\MyBoxWidth \linewidth}{
 $U^{G*H}_\lambda := \{ A \ | \ A \in V[G] \text{ and } j^{G*H} " H_\lambda[G] \in j^{G*H} (A)  \}$
}
\end{equation} 
is a $V[G]$-normal ultrafilter.  Then from the point of view of $V[G]$, the poset $j(\mathbb{P})/G$ forces that $\langle \dot{U}^{G*\dot{H}}_\lambda \ | \ \lambda < \delta \rangle$ is a tower of $V[G]$-normal measures (external to $V[G]$ of course).  Then in $V[G]$, let $\langle I_\lambda \ | \ \lambda < \delta \rangle$ be the tower of normal ideals derived from the name $\langle \dot{U}^{G*\dot{H}}_\lambda \ | \ \lambda < \delta \rangle$ as in Definition  \ref{def_TowerDerivedFromName} (here $V[G]$ is playing the role of $V$ and $j(\mathbb{P})/G$ is playing the role of $\mathbb{Q}$ from Definition \ref{def_TowerDerivedFromName}).  

\begin{definition}\label{def_TowerInducedBy_j}
The tower $\vec{I} \in V[G]$ described in the last paragraph will be called the \emph{tower induced by $j$}.
\end{definition}
We caution that if $j_{\vec{U}}: V \to N_{\vec{U}}$ is an embedding by a tower of $V$-normal measures, $j_{\vec{U}} \upharpoonright \mathbb{P}: \mathbb{P} \to j_{\vec{U}}(\mathbb{P})$ is a regular embedding, $G$ is $(V, \mathbb{P})$-generic, and $\vec{I} \in V[G]$ is the tower induced by $j_{\vec{U}}$ as in Definition \ref{def_TowerInducedBy_j}, then for each $\lambda < \delta$ it will \emph{NOT} in general be the case that the dual of $I_\lambda$ extends $U_\lambda$.  This is because of the way that the measure $U^{G*H}_\lambda$ is defined in (\ref{eq_U_GstarH}):  the measure $U_{G*H}$ concentrates on elementary substructures of $H_\lambda[G]$, \emph{NOT} on elementary substructures of $H_\lambda$.  This is only a minor technical issue, however; generally $N_{\vec{U}} \cap j^{G*H}_{\vec{U}} " H_\lambda[G] = j_{\vec{U}} " H_\lambda$, and it follows that for every $\lambda < \delta$ there are $U^{G*H}_\lambda$-many $M \prec H_\lambda[G]$ such that $M \cap V \in V$ (see Corollary \ref{cor_CreatingTowers} for the use of derived towers in this setting).       

\section{The weak Chang property and ideals which bound their completeness}\label{sec_TowersAndWeakChang}

In this section we discuss presaturation of towers and some concepts introduced by the first author in \cite{Cox_MALP} which will be used in the proofs of Theorems \ref{thm_RP_Ideal} and \ref{thm_SRP_TP_GIS}.  These concepts are related to Chang's Conjecture, bounding by canonical functions, and saturation.  For the reader's convenience all relevant proofs are included here.

A tower of height $\delta$ is called \emph{presaturated} iff $\delta$ always remains a regular cardinal in generic extensions by the tower.  Such a tower is always precipitous and $ult(V,G)$ is closed under $< \delta$ sequences from $V[G]$ (see Proposition 9.2 of \cite{MattHandbook}).  Woodin showed that if $\delta$ is a Woodin cardinal, there are several stationary towers of height $\delta$ which are presaturated.  We use the following weakening of presaturation introduced in \cite{Cox_MALP}:

\begin{definition}\label{def_weakChangProperty}
(Cox~\cite{Cox_MALP}) A tower of inaccessible height $\delta$ has the \emph{weak Chang property} iff whenever $G$ is generic for the tower, then $\delta$ is an element of the wellfounded part of $ult(V,G)$ and is regular in $ult(V,G)$ (though not necessarily in $V[G]$).
\end{definition}
\begin{lemma}
Let $\mu = \lambda^+$.  If a tower $\vec{I}$ of height $\delta$ concentrates on $\Gamma:=\{ M \ | \ |M| = \lambda \subset M   \}$, then $\vec{I}$ has the weak Chang property iff it forces that $j_G(\mu) = \delta$.
\end{lemma} 
\begin{proof}
The fact that $\vec{I}$ concentrates on $\Gamma$ implies that $\mu$ will be the critical point of $j_G$ and $j_G(\mu) \supseteq \delta$ for any generic $G$ (see \cite{MattHandbook}).  Since $j_G(\mu)$ is the successor of $\lambda$ in $ult(V,G)$, the equivalence follows easily.   
\end{proof}

\begin{corollary}\label{cor_PresatImpliesWeakChang}
Let $\mu = \lambda^+$ and assume $\vec{I}$ is a tower of height $\delta$ which concentrates on $\Gamma:=\{ M \ | \ |M| = \lambda \subset M   \}$.  If $\vec{I}$ is presaturated then it satisfies the weak Chang Property.
\end{corollary}

For the next lemma we will use the following definition, which is also related to saturation properties of ideals (see \cite{Cox_MALP}):

\begin{definition}\label{def_IdealBoundsCompleteness}
(Cox~\cite{Cox_MALP}) Let $J$ be a normal ideal over $\wp(H)$ where $\mu = completeness(J) \subseteq H$.  We say \emph{$J$ bounds its completeness} iff for every $f: \mu \to \mu$:  there are $\breve{J}$-many $M$ such that $otp(M \cap ORD) > f(M \cap \mu)$.
\end{definition}

\begin{lemma}\label{lem_LessDelta}
Suppose $\vec{I} = \langle I_\theta \ | \ \theta \in U \subset \delta \rangle$ is a tower of inaccessible height $\delta$, has completeness $\mu:= \lambda^+$, concentrates on $\{ M \ | \ |M|=\lambda \subset M  \}$, and has the weak Chang property.  Then:
\begin{enumerate}
 \item\label{item_LessDelta} For every generic $G$ and every $\theta < \delta$: $j_{proj(G,\theta)}(\mu) < \delta$.
 \item\label{item_SomeIdealBounds} There is a restriction of some ideal in the tower which bounds its completeness.
\end{enumerate}
\end{lemma}
\begin{proof}
Part \ref{item_LessDelta}:    Suppose not; let $\mu:= \lambda^+$ and let $G$ and $\theta$ be such that $j_{proj(G,\theta)}(\mu) \ge \delta$.  By assumption, $j_G(\mu) = \delta$; so in fact $\delta  = j_{proj(G,\theta)}(\mu)  = \{ [f]_{proj(G,\theta)} \ | \ f \in ({}^{Z_\theta} \mu)^V  \}$.  By Fact \ref{fact_TowerFacts}, $\{ [f]_{proj(G,\theta)} \ | \ f \in ({}^{Z_\theta} \mu)^V  \}$ is an element of $ult(V,G)$; moreover
\begin{displaymath}
|\{ [f]_{proj(G,\theta)} \ | \ f \in ({}^{Z_\theta} \mu)^V  \}|^{ult(V,G)} \le |({}^{Z_\theta} \mu)^V|^{ult(V,G)} \le |({}^{Z_\theta} \mu)^V|^V < \delta
\end{displaymath}
(by inaccessibility of $\delta$ in $V$).  This contradicts that $\delta$ is regular in $ult(V,G)$.

Part \ref{item_SomeIdealBounds}:  By part \ref{item_LessDelta} with $\theta:= \mu$, there is a condition $(\alpha, A)$ in the tower which decides the value of $j_{proj(\dot{G}, \mu)}(\mu)$ as some $\eta < \delta$.  Without loss of generality, assume:
\begin{equation}
\eta < \alpha
\end{equation}
We show that $I_\alpha \upharpoonright A$ bounds its completeness; a similar argument shows that $I_\beta \upharpoonright A^{V_\beta}$ bounds its completeness for every $\beta \in [\alpha, \delta)$.

Let $f: \mu \to \mu$.  Suppose for a contradiction that there were some $A' \subseteq A$ such that $A'$ is $I_\alpha$-positive and for every $M \in A'$: $otp(M \cap ORD) \le f(M \cap \mu)$ (note also that $M \cap ORD = M \cap \alpha$ for all $M \in A'$).  Let $G$ be generic for the tower with $(\alpha, A') \in G$.  Then $A' \in proj(G, V_\alpha)$ and so:
\begin{displaymath}
\eta < \alpha = [M \mapsto otp(M \cap \alpha)]_{proj(G, V_\alpha)} \le [M \mapsto f(M \cap \mu)]_{proj(G, V_\alpha)}
\end{displaymath}
Now $f$ maps into $\mu$, so:
\begin{displaymath}
[f]_{proj(G,\mu)} < j_{proj(G,\mu)}(\mu) = \eta
\end{displaymath}
So by  part \ref{item_k_maps} of Fact \ref{fact_TowerFacts}, $[f]_{proj(G,\mu)}$ is not moved by $k_{proj(G,\mu), proj(G,V_\alpha)}$:
\begin{displaymath}
[f]_{proj(G, \mu)} = k_{proj(G,\mu), proj(G,V_\alpha)}([f]_{proj(G, \mu)}) = [M \mapsto f(M \cap \mu)]_{proj(G,V_\alpha)}  
\end{displaymath} 
But this implies $\eta < \eta$, a contradiction.     
\end{proof}

\section{\texorpdfstring{RP and towers on $GIC_{\omega_1}$}{RP and towers on GIComega1}}\label{sec_RP_towers}

In this section we prove Theorems \ref{thm_MainGIC}, \ref{thm_NoDefinablePrecipGIC}, and  \ref{thm_PFAplus_impliesNonPresatGIC}.  

\subsection{Proof of Theorem \ref{thm_MainGIC}}

Theorem \ref{thm_MainGIC} follows from Corollary \ref{cor_PresatImpliesWeakChang}, Lemma \ref{lem_LessDelta}, and the following:
\begin{theorem}\label{thm_RP_Ideal}
Assume $RP([\omega_2]^\omega)$.  Then there is no ideal $I$ which bounds its completeness and concentrates on $GIC_{\omega_1}$.
\begin{com}
ALSO:  Can you get a similar result for $\omega_2$-guessing models?  What kind of isomorphism theorem can you get for transitive models $M$, $N$ such that $M \cap H_{\omega_2} = N \cap H_{\omega_2}$ and both are $\omega_2$-guessing?  To run Matteo's argument, you would need to know, I think, that $M \cap [M]^{\omega_1}$ is cofinal in $N \cap [N]^{\omega_1}$ and vice-versa...
\end{com}

\end{theorem}
\begin{proof}
Todorcevic proved that $RP([\omega_2]^\omega)$ implies $2^\omega \le \omega_2$ (see Theorem 37.18 of \cite{MR1940513}).  If CH holds, then for every $\theta \ge \omega_2$ the set of $\omega_1$-guessing submodels of $H_{\theta}$ is nonstationary (see \cite{Viale_GuessingModel}) and the theorem holds trivially.  

So suppose from now on that $2^\omega = \omega_2$.  Suppose for a contradiction that $I$ is a normal ideal concentrating on some stationary subset $S$ of $GIC_{\omega_1}$ (at some $H_\theta$), and that $I$ bounds its completeness (which is $\omega_2$).    Without loss of generality we assume that for every $M \in S$, $M \prec (H_\theta, \in, \Delta, \phi)$ where $\phi$ is some enumeration of the reals.  For each $\alpha \in proj(S, \omega_2)$ let $T(\alpha)$ be the collection of all transitive sets of the form $H_M$, where $M \in S$ and $M \cap \omega_2 = \alpha$.  

Let $\bar{I}$ be the projection of $I$ to $\omega_2$.  
\begin{NonGlobalClaim}\label{clm_AtLeastOmega2}
For $\bar{I}$-measure-one many $\alpha$, $s_\alpha := sup \{ ht(H) \ | \ H \in T(\alpha) \}$ is at least $\omega_2$.
\end{NonGlobalClaim}
\begin{proof}
(of Claim \ref{clm_AtLeastOmega2}):  Suppose not; so there is some $S'$ which is $I$-positive and for every $M \in S'$, $s_{M \cap \omega_2} < \omega_2$.  Let $f: \omega_2 \to \omega_2$ be defined by sending $\alpha \mapsto s_\alpha $ if $s_\alpha  < \omega_2$, and $f(\alpha) = 0$ otherwise.  Since $I$ bounds its completeness, there is some $C \in \breve{I}$ such that for every $M \in C$:  $otp(M \cap ORD) = ht(H_M) > f(M \cap \omega_2)$.  Then for every $M \in C \cap S'$:  $f(M \cap \omega_2) = s_{M \cap \omega_2}   < ht(H_M)$.  Fix any $\hat{M} \in C \cap S'$ and let $\hat{\alpha} := \hat{M} \cap \omega_2$; then
\begin{equation}
s_{\hat{\alpha}}   < ht(H_{\hat{M}})
\end{equation}
yet $H_{\hat{M}} \in T(\hat{\alpha})$; this is clearly a contradiction to the definition of $s_{\hat{\alpha}}$.
\end{proof}

Fix any $\alpha$ such that $s_\alpha = \omega_2$, and let $W := \bigcup T(\alpha)$.  Now $S \subseteq GIC_{\omega_1}$, so by Theorem \ref{thm_IsomIS} whenever $H$ and $H'$ are elements of $T(\alpha)$ and $ht(H) < ht(H')$, then $H$ is a hereditary initial segment of $H'$; this implies that $W$ is a transitive ZFC model (of height $\omega_2$).  Since $H \in IC_{\omega_1}$ for every $H \in T(\alpha)$, then:
\begin{equation}\label{eq_AllInitSegsClub}
\parbox{\MyBoxWidth \linewidth}{
For every $\beta < \omega_2$, $W \cap [\beta]^\omega$ contains a club.
}
\end{equation}
To see why (\ref{eq_AllInitSegsClub}) holds: let $\beta < \omega_2$.  Pick an $H \in T(\alpha)$ such that $\beta < H \cap ORD$, and let $\langle N_\xi \ | \ \xi < \omega_1 \rangle$ witness that $H \in IC_{\omega_1}$.  Then $\{ N_\xi \cap \beta \ | \ \xi < \omega_1 \}$ is a closed unbounded subset of $[\beta]^\omega$, and each $N_\xi \cap \beta$ is an element of $H \subset W$.  

Now $\mathbb{R} \cap W = \phi[\alpha]$; in particular $\mathbb{R} - W \ne \emptyset$.  By Theorem \ref{thm_Velick},  $S:=[\omega_2]^\omega - W$ is stationary (in fact projective stationary).  By $RP([\omega_2]^\omega)$, there is a $\beta < \omega_2$ such that $S \cap [\beta]^\omega$ is stationary.\footnote{This uses the fact that $\{ \beta \ | \ \omega_1 \le \beta < \omega_2  \}$ is a club subset of $[\omega_2]^{\omega_1}$ and that  $RP([\omega_2]^\omega)$ implies the following apparently stronger statement (see Theorem 3.1 of Feng-Jech~\cite{MR977476}):  for every stationary $S \subset [\omega_2]^\omega$, there are \emph{stationarily many} $Z \subset [\omega_2]^{\omega_1}$ such that $\omega_1 \subset Z$ and $S \cap [Z]^\omega$ is stationary.}  This contradicts (\ref{eq_AllInitSegsClub}). 
\end{proof}

\subsection{Proof of Theorem \ref{thm_NoDefinablePrecipGIC}}
Now we prove Theorem \ref{thm_NoDefinablePrecipGIC}; that is, if RP is omitted from the hypothesis of Theorem \ref{thm_MainGIC}, the Isomorphism Theorem for $GIC_{\omega_1}$ prevents precipitous towers on $GIC_{\omega_1}$ which are \emph{definable}.
\begin{proof}
If CH holds there are no $G_{\omega_1}$ structures so the theorem is trivial.  So assume $2^\omega = \omega_2$.  Suppose $\vec{I} = \langle I_\theta \ | \ \theta  \in U \rangle$ were such a tower.  By Lemma 4.3 of Burke~\cite{MR1472122}, a precipitous tower is not an element of the generic ultrapower.\footnote{ If the generic embedding moves $\delta$, which is always the case if the tower concentrates on $\{ M \ | \ |M|= \lambda \subset M  \}$.}  Since we are assuming the tower is definable over $V_\delta$, to obtain a contradiction it suffices to show that $V_\delta$ is an element of some generic ultrapower by the tower.  Let $G$ be generic for the tower; by Fact \ref{fact_TowerFacts}, $H_\theta^V \in ult(V,G)$ for each $\theta \in U$.  Also, since $\vec{I}$ concentrates on $GIC_{\omega_1}$ and $\omega_1 < cr(j_G)$, then by the Los Theorem, $ult(V,G) \models$ ``$H_\theta^V \in GIC_{\omega_1}$'' for each $\theta < \delta$.  Set $\kappa:= \omega_2^V$.  By Theorem  \ref{thm_IsomIS}, for each $\theta < \delta$, $ult(V,G)$ believes there is at most one $H$ such that $H$ is a transitive $ZF^-$ model of height $H \cap ORD$ such that $\mathbb{R}^{ult(V,G)} \cap H = \mathbb{R}^{ult(V,G)} \cap (H_{\omega_2}^V)$.  Thus when $ult(V,G)$ takes the union of all such models of height $< \delta$, the result is $V_\delta$ (since for each $\theta \in U$, $H_\theta \in ult(V,G)$ is such an $H$).  So $V_\delta \in ult(V,G)$ and we have a contradiction.     
\end{proof}

\subsection{Proof of Theorem \ref{thm_PFAplus_impliesNonPresatGIC}}

Finally we prove Theorem \ref{thm_PFAplus_impliesNonPresatGIC}.  It is well-known that either $MM$ or $PFA^+$ implies $RP$; so Theorem \ref{thm_PFAplus_impliesNonPresatGIC} will follow from Theorem \ref{thm_MainGIC} and the following: 
\begin{lemma}
Assume PFA and let $\delta$ be inaccessible.  Then there is a tower of height $\delta$ which concentrates on $GIC_{\omega_1}$.
\end{lemma}
\begin{proof}
In \cite{VW_ISP} it was shown that PFA implies that $GIC_{\omega_1} \cap \wp_{\omega_2}(H_\theta)$ is stationary for all regular $\theta \ge \omega_2$.  

For each $\lambda < \delta$ set $Z_\lambda := \{ M \cap H_\lambda \ | \ M \in GIC_{\omega_1} \cap \wp_{\omega_2}(V_\delta)  \}$ and set $I_\lambda := $ the projection of $NS \upharpoonright GIC_{\omega_1} \cap \wp_{\omega_2}(V_\delta)$ to a normal ideal on $Z_\lambda$.  It is straightforward to check that a sequence of ideals defined in this way is a tower.  By Lemma \ref{lem_ClassesProject}, each $Z_\lambda \subset GIC_{\omega_1} \cap \wp_{\omega_2}(H_\lambda)$.   

Alternatively, once can check that the sequence $\langle Z_\lambda \ | \ \lambda < \delta \rangle$ satisfies Lemma 9.49 of \cite{MattHandbook}, and then use Burke's ``stabilization'' technique to produce a tower of ideals concentrating on the $Z_\lambda$s.  It is not clear whether this yields the same tower as the previous paragraph.
\end{proof}

\section{\texorpdfstring{SRP and towers on $GIS_{\omega_1}$}{SRP and towers on GISomega1}}\label{sec_SRP_towers}

In this section we prove Theorem \ref{thm_MainGIS}.  The Tree Property at $\kappa$ ($TP(\kappa)$) is the statement that every tree of height $\kappa$ and width $< \kappa$ has a cofinal branch.  Theorem \ref{thm_MainGIS} follows from Corollary \ref{cor_PresatImpliesWeakChang}, Lemma \ref{lem_LessDelta}, and the following theorem:
\begin{theorem}\label{thm_SRP_TP_GIS}
Assume $SRP(\omega_2)$ and $TP(\omega_2)$.  Then there is no ideal concentrating on  $GIS_{\omega_1}$ which bounds its own completeness.
\end{theorem}

First, note that SRP implies that $NS_{\omega_1}$ is saturated and that $2^\omega = \omega_2$ (see Chapter 37 of \cite{MR1940513}).

Suppose for a contradiction that $I$ concentrates on some stationary $S \subseteq GIS_{\omega_1}$ and bounds its own completeness (which is $\omega_2$).  Without loss of generality we can assume that for every $M \in S$: $M \prec (H_\theta, \in, \phi)$ where $\phi$ is some wellorder of the reals and $H_\theta$ is the support of $I$.  For each $M \in S$ let $T_M \subseteq \omega_1$ be the stationary set witnessing that $M \in IS_{\omega_1}$.  For each $\alpha \in proj(S, \omega_2)$ define $T(\alpha):= \{ H_M \ | \ M \in S \text{ and } \alpha = M \cap \omega_2 \}$; the downward closure of $T(\alpha)$ under the hereditary initial segment relation\footnote{i.e. the nodes of $T(\alpha)$ consists of transitive models of the form $H_M$ and models of the form $(H_{\lambda})^{H_M}$ where $\lambda \in REG^{H_M}$.} forms a tree of height $\le \omega_2$.
\begin{NonGlobalClaim}\label{clm_ThinTree}
For each $\alpha \in proj(S, \omega_2)$, the tree $T(\alpha)$ has width $< \omega_2$.
\end{NonGlobalClaim}
\begin{proof}
Fix such an $\alpha$ and a level $\eta < \omega_2$ of the tree $T(\alpha)$.  Note that if $H$ is at the $\eta$-th level, then there is some $M \in S$ such that $H = (H_\lambda)^{H_M}$ where $\lambda$ is the $\eta$-th regular cardinal of $H_M$ (or $\lambda = H_M \cap ORD$).  Without loss of generality we assume $\eta \ge 2$; then it is straightforward to show that $\sigma_M [H] = M \cap H_{\sigma_M(\lambda)} \in GIS_{\omega_1}$ and that the set $T_M$---which witnesses that $M \in IS_{\omega_1}$---also witnesses that $M \cap H_{\sigma_M(\lambda)} \in IS_{\omega_1}$.  
\begin{com}
SHOULD PUT THAT IN A LEMMA?
\end{com}

Suppose for a contradiction that level $\eta$ had at least $\omega_2$-many distinct nodes $\langle H_\xi \ | \ \xi < \omega_2 \rangle$, and say $T_\xi \subset \omega_1$ witnesses that $H_\xi \in IS_{\omega_1}$.  Note all the $H_\xi$s have the same intersection with the reals (namely $\phi[\alpha]$; so they have the same intersection with $H_{\omega_1}$ as well).  For any distinct pair $\xi$ and $\xi'$, since $H_{\xi} \ne H_{\xi'}$ then $T_\xi \cap T_{\xi'}$ is nonstationary by Theorem \ref{thm_IsomIS}.  But then $\{ T_\xi \ | \ \xi < \omega_2 \}$ would be an $\omega_2$-sized antichain for $NS_{\omega_1}$, contradicting the fact that $NS_{\omega_1}$ is saturated.  
\end{proof}
Let $\bar{I}$ be the projection of $I$ to $\omega_2$.
\begin{NonGlobalClaim}\label{clm_TreeHeightOmega2}
For $\bar{I}$-measure one many $\alpha < \omega_2$, the tree $T(\alpha)$ has height $\omega_2$.
\end{NonGlobalClaim}
\begin{proof}
The proof of Claim \ref{clm_AtLeastOmega2} can be repeated verbatim.
\end{proof}

So by Claims \ref{clm_ThinTree} and \ref{clm_TreeHeightOmega2}, for $\bar{I}$-measure-one many $\alpha< \omega_2$, $T(\alpha)$ is a thin tree of height $\omega_2$.  Fix such an $\alpha$.  By $TP(\omega_2)$, $T(\alpha)$ has a cofinal branch.  The union of this branch is a transitive ZFC model $W$ of height $\omega_2$.  Now $W \cap \mathbb{R} = \phi[\alpha]$; so in particular $\mathbb{R} - W \ne \emptyset$ and so by Corollary \ref{cor_ToVelick}, $S := [\omega_2]^\omega - W$ is projective stationary.  Let $\tilde{S}:= \{  N \in [H_{\omega_2}]^\omega \ | \ N \cap \omega_2 \in S \}$; then $\tilde{S}$ is projective stationary in $[H_{\omega_2}]^\omega$.\footnote{This is standard.  If $T \subset \omega_1$ is stationary, then $S_T:= \{ Z \in [\omega_2]^\omega \ | \ Z \cap \omega_1 \in T  \}$ is stationary by the projective stationarity of $S$.  Let $\mathcal{A}$ be a structure on $H_{\omega_2}$ in a countable language; we need to find an elementary substructure of $\mathcal{A}$ whose intersection with $\omega_1$ is in $T$.  Pick any $Z \in S_T$ such that $Sk^{\mathcal{A}}(Z) \cap \omega_2 = Z$ (this holds for all but nonstationarily many $Z \in S_T$).  Then $Sk^{\mathcal{A}}(Z) \in \tilde{S}$ is the model we seek.}  

By $SRP(\omega_2)$, there is a continuous chain $\langle N_\xi \ | \ \xi < \omega_1 \rangle$ of elementary substructures of $H_{\omega_2}$ such that $N_\xi \in \tilde{S}$ for every $\xi< \omega_1$.  Let $Z:= \bigcup_{\xi < \omega_1} (N_\xi \cap \omega_2)$; it can easily be shown that $Z$ is an ordinal;\footnote{Let $\beta \in (\omega_2 \cap \bigcup_{\xi < \omega_1} N_\xi)$ and let $\zeta < \beta$; we need to see that $\zeta \in \bigcup_{\xi < \omega_1} N_\xi$.  Let $\xi^*$ be such that $\beta \in N_{\xi^*}$; since $N_{\xi^*} \prec H_{\omega_2}$ there is a bijection $f: \omega_1 \to \beta$ such that $f \in N_{\xi^*}$.  Let $\xi' := f^{-1}(\zeta)$, and pick any $\xi^{''}$ such that $\xi' \in N_{\xi^{''}}$ and $\xi^* \le \xi^{''}$; then both $f$ and $\xi'$ are elements of $N_{\xi^{''}}$, so $\zeta \in N_{\xi^{''}}$.} in particular $Z$ is an \emph{element} of some $H \in T(\alpha)$.  This implies that $proj(H,Z):= \{ a \cap Z \ | \ a \in H \cap [H]^\omega   \}$ is a subset of $H$.  Also, $proj(H,Z)$ is stationary in $[Z]^\omega$, since $H \cap [H]^\omega$ is stationary by assumption.  Moreover, $C:=\{  N_\xi \cap \omega_2 \ | \ \xi < \omega_1 \} \subseteq S$ and $C$ is a club subset of $[Z]^\omega \cap S$.  So $proj(H,Z) \cap C$ is nonempty.  But $proj(H,Z)   \subseteq H \subset W$; yet $C \subseteq S$ and $S \cap W = \emptyset$, a contradiction.  This completes the proof of Theorem \ref{thm_SRP_TP_GIS}.

\section{\texorpdfstring{Consistency of $MM^+$ with a precipitous tower on $GIC_{\omega_1}$}{Consistency of MM-plus with a precipitous tower on GIC}}\label{sec_Consist_MM_GIC_Precip}

Now we prove Theorem \ref{thm_Consistent_GIC_Precip}.  First, we need a ``tower'' version of Proposition 7.13 from \cite{MattHandbook}.  

\begin{theorem}\label{thm_DualityTowers}
(modification of Proposition 7.13 from \cite{MattHandbook} for towers).

Suppose $\mathbb{Q} \in V$ is a poset and $1_{\mathbb{Q}} \Vdash$ ``$\delta$ remains inaccessible, $\langle \dot{U}_\lambda \ | \ \lambda < \delta \rangle$ is a tower of $V$-normal measures, and $ult(V, \dot{\vec{U}})$ is wellfounded''.  Suppose also that in $V$ there are functions $Q$, $h$, and for each $q \in \mathbb{Q}$ a function $f_q$ such that:
\begin{itemize}
 \item $Q$, $h$, and each $f_q$ each have bounded support in $V_\delta$;
 \item For every $\mathbb{Q}$-generic object $H$:
  \begin{itemize}
   \item $[Q]_{\dot{\vec{U}}_H} = \mathbb{Q}$; 
   \item $[h]_{\dot{\vec{U}}_H} = H$;
   \item For every $q \in \mathbb{Q}$:  $[f_q]_{\dot{\vec{U}}_H} = q$.
  \end{itemize}
\end{itemize}   
If $\vec{I} \in V$ is the tower derived from the name $\dot{\vec{U}}$ as in Definition \ref{def_TowerDerivedFromName}, then $\mathbb{P}_{\vec{I}}$ is precipitous, forcing equivalent to $\mathbb{Q}$, and generic ultrapowers by $\vec{I}$ are exactly those maps of the form $j_{\dot{\vec{U}}_H}: V \to ult(V,\dot{\vec{U}}_H)$ where $H$ is $(V, \mathbb{Q})$-generic. 
\end{theorem}
\begin{proof}
First, we note that if $\vec{I}$ is a tower where each ideal $I_\lambda \subset \wp(Z_\lambda)$, then the poset  $\mathbb{P}_{\vec{I}}$ (as defined in section \ref{sec_TowersMeasuresIdeals}) is forcing equivalent to the poset obtained as follows:  Define an equivalence relation on $\mathbb{P}_{\vec{I}}=\{ (\lambda, S) \ | \ \lambda < \delta \text{ and } S \in I_\lambda^+  \}$ by:
\begin{equation}
\parbox{\MyBoxWidth \linewidth}{
$(\lambda, S) \simeq (\beta, T)$ iff $S^{Z_\eta} \bigtriangleup T^{Z_\eta} \in I_\eta$ for some (equivalently:  every) $\eta \ge max(\lambda, \beta)$.
}
\end{equation}

Let $\mathbb{P}'_{\vec{I}} := \mathbb{P}_{\vec{I}} / \simeq$ and partially order $\mathbb{P}'_{\vec{I}}$ in the natural way inherited from the partial ordering of $\mathbb{P}_{\vec{I}}$.\footnote{Another way to view the poset $\mathbb{P}'_{\vec{I}}$ is to consider the directed system of ``canonical liftings'' $\iota_{\lambda, \lambda'}: \wp(Z_\lambda)/I_\lambda \to \wp(Z_{\lambda'})/I_{\lambda'}$ (for $\lambda \le \lambda'$) defined by $[S]_{I_\lambda} \mapsto [S^{Z_{\lambda'}}]_{I_{\lambda'}}$.  Then $\mathbb{P}'_{\vec{I}}$ is the direct limit of this system.}

Now let $\vec{I}$ be the tower derived from the name $\dot{\vec{U}}$ as in the statement of the theorem.  Similarly to the way Proposition 7.13 from \cite{MattHandbook} is proved, we define a map $\phi:  \mathbb{P}'_{\vec{I}} \to ro(\mathbb{Q})$ by:
\begin{equation}
[(\lambda,S)]_{\simeq} \mapsto || j_{\dot{\vec{U}}_H} " (\bigcup Z_\lambda) \in j_{\dot{\vec{U}}_H}(S)   ||_{ro(\mathbb{Q})}
\end{equation}
It is straightforward to check that this map is well-defined and preserves order and incompatibility.  Further, identifying $\mathbb{Q}$ with its isomorphic copy in $ro(\mathbb{Q})$, the assumptions of the theorem imply that $\mathbb{Q} \subseteq range(\phi)$:  given $q \in \mathbb{Q}$, let $f_q$ be as in the statement of the theorem, and let $\lambda_q < \delta$ be the support of $f_q$ (and without loss of generality assume $\lambda_q$ is also greater than the support of $h$ and $Q$).  Then $\phi$ maps the condition $[ (\lambda_q, \{ M \in Z_{\lambda_q} \ | \ f_q(M) \in h(M)  \} )     ]_{\simeq}$ to $q$.  

Finally, we sketch why generic ultrapowers by $\vec{I}$ are exactly those embeddings of the form $j_{\dot{\vec{U}}_H}$ for some $H$ which is $(V, \mathbb{Q})$-generic.  First, suppose $G$ is $(V, \mathbb{P}'_{\vec{I}})$-generic, and let $H$ be the $(V, \mathbb{Q})$-generic obtained from $G$ and the map $\phi$ (or vice-versa; the argument is similar either way).  Define a map $\ell: ult(V, \dot{\vec{U}}_H) \to ult(V,G)$ by:  $[f]_{\dot{\vec{U}}_H} \mapsto [f]_G$.  Then one can show that $\ell$ maps onto $ult(V, G)$ and that $j_{G} = \ell \circ j_{\dot{\vec{U}}_H}$.  

\end{proof}

\begin{corollary}\label{cor_CreatingTowers}
Suppose $\vec{U} \in V$ is a tower of normal ultrafilters of inaccessible height $\delta$ and $j_{\vec{U}}: V \to_{\vec{U}} N_{\vec{U}}$ is the ultrapower.  Suppose $\mathbb{P} \in V_\delta$ and that $j_{\vec{U}} \upharpoonright \mathbb{P}: \mathbb{P} \to j_{\vec{U}}(\mathbb{P})$ is a regular embedding.  Let $G$ be $(V, \mathbb{P})$-generic, and let $\vec{I} \in V[G]$ be the tower of height $\delta$ induced by  $j_{\vec{U}}$ as in Definition \ref{def_TowerInducedBy_j}.  

Then in $V[G]$:  $\vec{I}$ is precipitous,  $\mathbb{P}_{\vec{I}}$ is forcing equivalent to $j_{\vec{U}}(\mathbb{P})/j_{\vec{U}} " G$, and generic ultrapowers of $V[G]$ by $\vec{I}$ are exactly those maps of the form $j_{\vec{U}}^{G*H}: V[G] \to N_{\vec{U}}[G][H]$ where $H$ is $j_{\vec{U}}(\mathbb{P})/ j_{\vec{U}} " G$-generic over $V[G]$.
\end{corollary}
\begin{proof}
Let $G$ be $(V, \mathbb{P})$-generic.  We check the conditions of Theorem \ref{thm_DualityTowers}; here $V[G]$ will play the role of the $V$ from Theorem \ref{thm_DualityTowers} and $j_{\vec{U}}(\mathbb{P})/j_{\vec{U}} " G$ will play the role of the $\mathbb{Q}$ from Theorem \ref{thm_DualityTowers}.

Work in $V[G]$.  For all $H$ which are $(V[G], j_{\vec{U}}(\mathbb{P})/j_{\vec{U}} " G)$-generic:  for every $\lambda < \delta$, there are $U_\lambda^{G*H}$-many $M'$ such that:
\begin{enumerate}
 \item $M' \cap V \in V$;\footnote{This holds for $U^{G*H}_\lambda$-many $M'$ because $N_{\vec{U}} \cap j^{G*H}_{\vec{U}} " H_\lambda[G] = j_{\vec{U}} " H_\lambda$ } denote this set $M$
 \item $V \models$ ``$M \cap \mathbb{P}$ is a regular subposet of $\mathbb{P}$''. 
\end{enumerate}
Since we assume $\mathbb{P} \in V_\delta$, there is some $\lambda_{\mathbb{P}} < \delta$ such that $\mathbb{P} \in H_{\lambda_{\mathbb{P}}}$.  Now consider the following functions defined in $V[G]$ on $A^\delta:= \{ M' \prec V_\delta[G] \ | \ M' \cap V \in V \text{ and } M' \cap \mathbb{P} \text{ is a regular subposet of } \mathbb{P} \}$:
\begin{itemize}
 \item $Q(M'):= \mathbb{P}/(G \cap M')$; note this equals $\mathbb{P} / (G \cap M' \cap H^V_{\lambda_{\mathbb{P}}}$)
 \item $h(M'):=$ the generic for $\mathbb{P}/(G \cap M')$ obtained from $G$ and the forcing equivalence between $\mathbb{P}$ and $(M' \cap \mathbb{P}) * (\mathbb{P}/(\dot{G} \cap M')$.  Note this only depends on $M' \cap H^V_{\lambda_{\mathbb{P}}}$. 
 \item For any $q \in j_{\vec{U}}(\mathbb{P})/G$:  note that $q \in V$ and there is some $f_q: V_\delta \to V$ with support $\lambda_q < \delta$ such that $q = [f_q]_{\vec{U}}$.  Then define (in $V[G]$) the function $f'_q$ by $M' \mapsto  f_q(M' \cap  H^V_{\lambda_q})$.
\end{itemize}

Note that each of these functions has bounded support in $V_\delta$ ($Q$ and $h$ have support $\lambda_{\mathbb{P}}$, and $f'_q$ has support $\lambda_q$).  It is straightforward to check that for every $H$ which is $(V[G], j_{\vec{U}}(\mathbb{P})/j_{\vec{U}} " G)$-generic:
 \begin{itemize}
  \item $[Q]_{\vec{U}^{G*H}} = j_{\vec{U}}(\mathbb{P})/j_{\vec{U}} " G$;
  \item $[h]_{\vec{U}^{G*H}} = H$
  \item For each $q \in j_{\vec{U}}(\mathbb{P})/j_{\vec{U}} " G$:  $[f_q]_{\vec{U}^{G*H}} = q$
 \end{itemize}
The conclusion then follows by Theorem \ref{thm_DualityTowers}.
\end{proof}

It is interesting to note that if $\vec{I} \in V[G]$ is as in Corollary \ref{cor_CreatingTowers} and $\delta$ is always moved by generic embeddings of $V$ by $\vec{I}$,\footnote{This is always the case if each $U_\lambda$ in the original tower concentrates on $\wp_\kappa(H_\lambda)$.} then  generic ultrapowers of $V[G]$ by $\vec{I}$ do not have $\mathbb{P}_{\vec{I}}$ as an element (by Lemma 4.3 of \cite{MR1472122}).  However these generic ultrapowers \emph{do} have a poset---namely $j_{\vec{U}}(\mathbb{P})/ j_{\vec{U}} " G$---which, from the point of view of $V[G]$, is forcing equivalent to $\mathbb{P}_{\vec{I}}$ (and all the generic ultrapowers even have a $V[G]$-generic for that poset).    

Now back to the proof of Theorem \ref{thm_Consistent_GIC_Precip}.  Suppose $\kappa$ is supercompact and $\delta > \kappa$ is inaccessible.  Let $Lav: \kappa \to V_\kappa$ be a Laver function for $\kappa$, and $\mathbb{P}$ the standard RCS iteration of length $\kappa$ which yields a model of Martin's Maximum as in \cite{MR924672};  this actually produces a model of $MM^{+\omega_1}$.  
In $V$ let $U$ be a normal measure on $\wp_\kappa(H_\eta)$ for some regular $\eta \ge \delta$ such that $j_U(Lav)(\kappa) = \dot{\mathbb{R}}_\delta$, where $\mathbb{R}_\delta$ is the poset from Theorem \ref{thm_GIC_persistent} and $\dot{\mathbb{R}}_\delta$ is the canonical $\mathbb{P}$-name for $(\mathbb{R}_\delta)^{V^{\mathbb{P}}}$.  Let $\vec{U}:= \langle U_\lambda \ | \ \lambda < \delta \rangle$ be the tower of normal measures produced from projections of $U$ to $\wp_\kappa(H_\lambda)$ for $\lambda < \delta$.  Let $j_{\vec{U}}: V \to N_{\vec{U}}$; recall $N_{\vec{U}}$ is closed under $< \delta$ sequences so in particular $j_{\vec{U}} \upharpoonright H_\lambda \in N_{\vec{U}}$ for every $\lambda < \delta$.  Since $\mathbb{P}$ has the $\kappa$-cc, then $j_{\vec{U}} \upharpoonright \mathbb{P} = id: \mathbb{P} \to j_{\vec{U}}(\mathbb{P})$ is a regular embedding, so the discussion before Definition \ref{def_TowerInducedBy_j} applies.  Fix some $G$ which is $(V,\mathbb{P})$-generic, and in $V[G]$ let $\vec{I}$ be the tower of ideals induced by $j_{\vec{U}}$ as in Definition \ref{def_TowerInducedBy_j}.  By Corollary \ref{cor_CreatingTowers}:
\begin{equation}
\parbox{\MyBoxWidth \linewidth}{
$\vec{I}$ is precipitous
}
\end{equation}

So we only have left to show that $\vec{I}$ concentrates on $GIC_{\omega_1}$.  First we note:

\begin{claim}\label{clm_j_infty_Lav_R}
$j_{\vec{U}}(Lav)(\kappa) = \dot{\mathbb{R}}_\delta$
\end{claim}
\begin{proof}
By standard arguments there is a $k: N_{\vec{U}} \to ult(V,U)$ such that $k \circ j_{\vec{U}} = j_U$ and $k \upharpoonright \delta = id$.  Now $\dot{\mathbb{R}}_\delta = j_U(Lav)(\kappa) = k \circ j_{\vec{U}}(Lav) (k(\kappa)) = k(j_{\vec{U}}(Lav)(\kappa))$; so $\dot{\mathbb{R}}_\delta \in range(k)$.  Recall from Theorem \ref{thm_GIC_persistent} that the poset $\mathbb{R}_\delta$ is always an element of $H_{\delta^+}$; so the canonical $\mathbb{P}$-name $\dot{\mathbb{R}}_\delta$ for $\mathbb{R}_\delta^{V^{\mathbb{P}}}$ is an element of $H^V_{\delta^+} = H^{ult(V,U)}_{\delta^+}$.  So $| \dot{\mathbb{R}}_\delta |^{ult(V,U)} = \delta$.  Then since $\dot{\mathbb{R}}_\delta \in range(k)$, we have $\delta = |\dot{\mathbb{R}}_\delta|^{ult(V,U)} \in range(k)$.  This implies that $cr(k) > \delta$  (equivalently, that $j_{\vec{U}}(\kappa) > \delta$) and that $k^{-1}(\dot{\mathbb{R}}_\delta) = \dot{\mathbb{R}}_\delta$.   
\end{proof}

Consider an arbitrary $H$ which is $(V[G], j_{\vec{U}}(\mathbb{P})/G)$-generic.  Let $H^*$ denote the $\kappa$-th component of $H$.  Now $N_{\vec{U}}[G][H^*] \models V_\delta[G] \in GIC_{\omega_1}$ because $H^*$ is $(N_{\vec{U}}[G], \mathbb{R}_\delta^{N_{\vec{U}}[G]})$-generic (note $V_\delta = V_\delta^{N_{\vec{U}}}$ because $N_{\vec{U}}$ is closed under $< \delta$ sequences from $V$).  Since $N_{\vec{U}}[G][H]$ is an outer model of $N_{\vec{U}}[G][H^*]$ with the same $\omega_1$, then Theorem \ref{thm_GIC_persistent} implies:

\begin{equation}\label{eq_VdeltaG_in_GIC}
N_{\vec{U}}[G][H] \models V_\delta[G] \in GIC_{\omega_1}
\end{equation}

By (\ref{eq_VdeltaG_in_GIC}) and (the transitivised variant of) Theorem \ref{lem_ClassesProject}:

\begin{equation}\label{eq_EveryLambdaLessDelta}
\parbox{\MyBoxWidth \linewidth}{
For every $V$-regular $\lambda \in [\kappa, \delta]$:  $N_{\vec{U}}[G][H] \models H_\lambda[G] \in GIC_{\omega_1}$.
}
\end{equation}






Since $j_{\vec{U}}^{G*H} \upharpoonright H_\lambda[G]$ is an element of $N_{\vec{U}}[G][H]$ for every $\lambda < \delta$ and the class $GIC_{\omega_1}$ is closed under isomorphism, then (\ref{eq_EveryLambdaLessDelta}) implies:
\begin{equation}\label{eq_InGIC}
\parbox{\MyBoxWidth \linewidth}{
For every $V$-regular $\lambda \in [\kappa,  \delta)$, $N_{\vec{U}}[G][H] \models$ ``$j^{G*H}_{\vec{U}} " H_\lambda[G]$ is an element of $GIC_{\omega_1}$''.
}
\end{equation} 
Since (\ref{eq_InGIC}) holds for arbitrary generic $H$, then by the definition of each $I_\lambda$:
\begin{equation}
\parbox{\MyBoxWidth \linewidth}{
For each $\lambda < \delta$, $I_\lambda$ concentrates on $GIC_{\omega_1}$.
}
\end{equation}
This concludes the proof of Theorem \ref{thm_Consistent_GIC_Precip}.

\section{Questions}\label{sec_Questions}
We end with some questions.

We proved that under $RP([\omega_2]^\omega)$, there is no presaturated tower which concentrates on $GIC_{\omega_1}$.  This suggests a couple of questions:
\begin{question}\label{question_Consistent_RP_Presat_GIS}
Is it consistent with $RP([\omega_2]^\omega)$ that there is a presaturated tower concentrating on $GIS_{\omega_1}$?  
\end{question}
\begin{question}
Is it consistent with ZFC to have a presaturated tower which concentrates on $GIC_{\omega_1}$?
\end{question}
One way to produce a presaturated tower on $GIS_{\omega_1}$ is to perform a ``Mitchell collapse'' so that an almost-huge cardinal becomes $\omega_2$; however $RP[\omega_2]^\omega$ fails in this model, so it does not provide an affirmative answer to Question \ref{question_Consistent_RP_Presat_GIS}.  

We also showed that $MM$ implies there is no presaturated tower on $GIS_{\omega_1}$, which suggests:
\begin{question}
Is it consistent with $MM$ that there is a presaturated tower concentrating on $GIU_{\omega_1}$?
\end{question}
\begin{question}
If the answer to either of the previous questions is ``yes'', can this tower be a \emph{stationary} tower?  Or any other kind of ``natural'' tower?
\end{question}
Finally, in Theorem \ref{thm_NoDefinablePrecipGIC} we showed there is no precipitous tower on $GIC_{\omega_1}$ which is \emph{definable} over $V_\delta$ (where $\delta$ is the height of the tower).  
\begin{question}
Suppose $NS_{\omega_1}$ is saturated.  Does this imply that there is no precipitous tower on $GIS_{\omega_1}$ which is \emph{definable} over $V_\delta$?  (Where $\delta$ is the height of the tower)
\end{question}

\end{document}